\documentclass[twoside, reqno]{amsart} 
\usepackage{xcolor}
\usepackage{amssymb} 
\usepackage{amsthm}
\usepackage{amsmath}
\usepackage{mathtools} 
\usepackage{epigraph}
\usepackage{microtype}
\usepackage{pgfplots}
\pgfplotsset{compat=1.18}

\usepackage{hyperref} 
\definecolor{bluey}{rgb}{0,0,0.6}
\hypersetup{colorlinks,
linkcolor=bluey,
citecolor=bluey,
urlcolor=bluey}

\newtheorem{theorem}{Theorem}
\newtheorem{lemma}{Lemma}

\makeatletter
\def\thmhead@plain#1#2#3{%
  \thmname{#1}\thmnumber{\@ifnotempty{#1}{ }\@upn{#2}}%
  \thmnote{ \the\thm@notefont(#3)}}
\let\thmhead\thmhead@plain
\makeatother

\begin{document}
\title[3AP-free permutations have no exponential growth rate]%
{3AP-free permutations have \\ no exponential growth rate}
\author{Boon Suan Ho}
\address{Department of Mathematics, National University of Singapore}
\email{\href{mailto:hbs@u.nus.edu}{hbs@u.nus.edu}}

\begin{abstract} 
Let $\theta(n)$ be the number of permutations of $\{1,\dots,n\}$ with no
$3$-term arithmetic progressions.
We prove that 
$\lim_{n\to\infty}\theta(n)^{1/n}$ does not exist.
\end{abstract}

\maketitle

\section{Introduction}
A permutation $(\pi_1,\dots,\pi_n)$ of $\{1,\dots,n\}\eqqcolon[n]$ 
is said to be \emph{$3$AP-free} if no indices $i<j<k$ satisfying
$\pi_i+\pi_k=2\pi_j$ exist. We write $\theta(n)$ for the number of
$3$AP-free permutations of $[n]$.
The study of $3$AP-free permutations of~$[n]$
dates back to at least 1973, when Entringer and Jackson \cite{EJ73}
asked if all permutations of $[n]$ contain a $3$AP 
in the \textsl{American Mathematical Monthly} problems section.
In the solutions, Simmons \cite{Sim75}
gave values of $\theta(n)$ for $1\le n\le20$ (some incorrect);
the initial values are $(\theta(1),\theta(2),\dots)
=(1, 2, 4, 10, 20, 48, 104, 282, 496, 1066, 2460, \dots)$.

Shortly after, Davis, Entringer, Graham, and Simmons \cite{DEGS77}
investigated $3$AP-free permutations of $[n]$,
as well as $k$AP-free permutations of $\mathbf N$ and $\mathbf Z$.
They proved the relations 
\begin{equation*}
\theta(2n)\ge2\theta(n)^2 
\quad\text{and}\quad
\theta(2n+1)\ge\theta(n+1)\theta(n)
\quad\text{for $n\ge1$,}
\end{equation*} 
as well as the bounds 
\begin{equation*}
2^{n-1}\le\theta(n)\le\Bigl\lfloor{n+1\over2}\Bigr\rfloor!\,
\Bigl\lceil{n+1\over2}\Bigr\rceil!
\quad\text{for $n\ge1$;}
\end{equation*}
they also showed that
$\theta(2^k)>{1\over2}c^{2^k}$ for $k\ge4$ where $c=(2\theta(16))^{1/16}\approx2.248$,
giving an improved exponential lower bound for $\theta(n)$ along powers of two.

After some time, these bounds were improved by
Sharma \cite{Sha09}, who proved the relation 
\begin{equation*}
\theta(n)\le21
\theta\Bigl(\Bigl\lfloor{n\over2}\Bigr\rfloor\Bigr)\,
\theta\Bigl(\Bigl\lceil{n\over2}\Bigr\rceil\Bigr)
\quad\text{for $n\ge3$,}
\end{equation*}
as well as the exponential lower and upper bounds
\begin{equation*}  
{n2^n\over10}\le\theta(n)\le{2.7^n\over21}
\quad\text{for $n\ge11$.}
\end{equation*}
LeSaulnier and Vijay \cite{LV11} then
improved the lower bound to $\theta(n)\ge{1\over2}c^n$ for $n\ge8$, 
where $c=(2\theta(10))^{1/10}\approx2.152$.

Nonetheless, much remains mysterious about $\theta(n)$; 
for example, it is not known whether
$\theta(n)\le\theta(n+1)$ for all $n$, despite this appearing rather obvious 
for small values of $n$. 
In 2004, Landman and Robertson \cite{LR04} asked in their book 
\textsl{Ramsey Theory on the Integers} if $\lim_{n\to\infty}\theta(n)^{1/n}$ exists.
Later in 2017, Correll and Ho \cite{CH17} presented a dynamic programming algorithm 
that enabled them to compute values of $\theta(n)$ up to $n=90$. Since then,
values of $\theta(n)$ up to $n=200$ have been computed by Alois Heinz 
(see OEIS \href{https://oeis.org/A003407}{A003407}). These values yield the estimates
\begin{align*}
\bigl(2\theta(160)\bigr)^{1/160}\approx2.20499 
&<\liminf_{n\to\infty}\theta(n)^{1/n} \\
&\le\limsup_{n\to\infty}\theta(n)^{1/n}<2.32721
\approx\bigl(21\theta(128)\bigr)^{1/128}.
\end{align*} 
This may seem to suggest that the limit converges, yet we have the following result:

\begin{theorem} \label{theorem:1}
$\lim_{n\to\infty}\theta(n)^{1/n}$
does not exist.
\end{theorem}

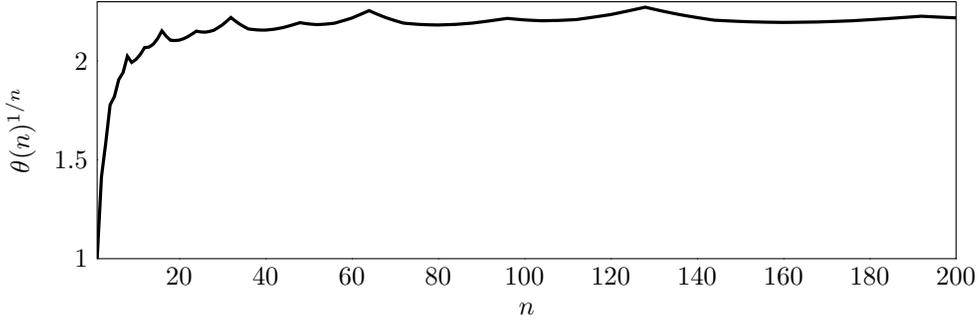
\begin{center}\vspace{0.5cm}\hspace*{-0.6cm} 
\begin{tikzpicture}
\begin{axis}[
  width=13cm,
  height=5cm,
  axis lines=box,
  xlabel={$n$},
  ylabel={$\theta(n)^{1/n}$},
  xmin=1, xmax=200,
  ymin=1.0, ymax=2.30,
  tick align=inside,
  tick style={black},
  xtick pos=bottom,
  ytick pos=left,
  major tick length=0pt,
  xmajorgrids=false,
  ymajorgrids=false,
]
\addplot[black, line width=1.2pt, mark=none]
  table [x index=0, y index=1, col sep=space] {theta_root_n1_200.dat};
\end{axis}
\end{tikzpicture}
\end{center}

\section{The proof}
The idea is to consider subsequences of $(\theta(n)^{1/n})_{n\ge1}$
of the form 
\begin{equation}
\bigl(b_t(m)\bigr)_{t\ge0}\coloneqq\bigl(\theta(m2^t)^{1/(m2^t)}\bigr)_{t\ge0}.
\end{equation}
We will show that these subsequences converge for each $m\ge1$, 
then we will pick two of these subsequences and use known values of $\theta(n)$
to prove that they converge to different limits. First, we need two lemmas:

\begin{lemma}[\cite{DEGS77}]
We have $2\theta(k)^2\le\theta(2k)$ for $k\ge1$.
\end{lemma}

\begin{proof}
Given two $3$AP-free permutations $(a_1,\dots, a_k)$ and 
$(b_1,\dots,b_k)$ of $[k]$,
both $$(2a_1,\dots,2a_k,2b_1-1,\dots,2b_k-1)
\;\;\text{and}\;\; 
(2b_1-1, \dots, 2b_k-1, 2a_1, \dots, 2a_k)$$
are $3$AP-free permutations of $(1,\dots,2k)$.
\end{proof}

\begin{lemma}[\cite{Sha09}, Theorem 2.8] \label{lemma:2}
For $n\ge3$, we have $\theta(n)\le21\theta(\lceil n/2\rceil)\theta(\lfloor n/2\rfloor)$. 
In particular, for $k\ge1$, we have $\theta(2k)\le21\theta(k)^2$.
\end{lemma}
These lemmas imply that
\begin{equation}
2\theta(k)^2\le\theta(2k)\le21\theta(k)^2
\quad\text{for $k\ge1$,}
\end{equation}
so taking $(2k)$-th roots gives
\begin{equation}
2^{1/(2k)}\theta(k)^{1/k}
\le\theta(2k)^{1/(2k)}
\le21^{1/(2k)}\theta(k)^{1/k}
\quad\text{for $k\ge1$.}
\end{equation}
Setting $k=m2^t$ with $m\ge1$ and $t\ge0$ thus gives
\begin{equation} \label{eq:star}
2^{1/(m2^{t+1})}b_t(m)\le b_{t+1}(m)\le21^{1/(m2^{t+1})}b_t(m).
\end{equation}
In particular, since $2^{1/(m2^{t+1})}>1$, the sequence $(b_t(m))_{t\ge0}$
is strictly increasing. 
Repeatedly applying the bound on the right in \eqref{eq:star} yields
\begin{equation}
b_t(m)\le b_0(m)\cdot21^{\sum_{j=1}^t{1\over m2^j}}\le b_0(m)\cdot21^{1/m},
\end{equation}
so $(b_t(m))_{t\ge0}$ is bounded above. Thus $(b_t(m))_{t\ge0}$
converges to some limit $L_m$. 

Now let us bound $L_m$. Applying \eqref{eq:star} with $t+r$ in place of $t$, we get
\begin{equation}
2^{1/(m2^{t+r+1})}b_{t+r}(m)\le b_{t+r+1}(m)\le21^{1/(m2^{t+r+1})}b_{t+r}(m).
\end{equation}
Iterating the left inequality $k$ times gives
\begin{equation}
b_{t+k}(m)\ge b_t(m)\prod_{r=1}^k2^{1/(m2^{t+r})}
\ge 2^{(1-2^{-k})/(m2^t)}b_t(m).
\end{equation}
Similarly, iterating the right inequality gives
\begin{equation*}
b_{t+k}\le 21^{(1-2^{-k})/(m2^t)}b_t(m).
\end{equation*}
Sending $k\to\infty$, we find that
\begin{equation} \label{eq:dagger} 
\bigl(2\theta(m2^t)\bigr)^{1/(m2^t)}
\le L_m
\le \bigl(21\theta(m2^t)\bigr)^{1/(m2^t)}
\end{equation}
for all $m\ge1$ and $t\ge0$.

To complete the proof, we show that $L_1\ne L_{75}$.
Indeed, since $64=1\cdot2^6$ and $75=75\cdot2^0$, \eqref{eq:dagger} implies that
$L_1\ge(2\theta(64))^{1/64}$ and $L_{75}\le(21\theta(75))^{1/75}$.
Since $\theta(64)=39911512393313043466768$ and $\theta(75)=30235147387260979648843264$
(see \cite{CH17} or OEIS \href{https://oeis.org/A003407}{A003407} for $\theta(n)$ values),
we have $L_1\ge(2\theta(64))^{1/64}\approx2.27953231299$
and $L_{75}\le(21\theta(75))^{1/75}\approx2.27703523933$.
Thus $L_1\ne L_{75}$, and we are done. 

\section{Discussion}
As a consequence of this result, $3$AP-free permutations provide an 
interesting and rare example of a combinatorial object that appears
(empirically) to be increasing (that is, writing $a_n$ for the number of objects of
size $n$, data suggests that $a_n\le a_{n+1}$ for all $n$), 
with reasonably close bounds $c^n\le a_n\le C^n$ for some constants $1<c<C$,
where the growth rate $\lim_{n\to\infty}a_n^{1/n}$ does not exist.

Why does this happen?
A common way to prove that such limits exist, though not the only way, is as follows:
Many combinatorial objects are such that you can combine smaller objects
to form bigger ones; this often leads to bounds such as $a_ma_n\le a_{m+n}$.
We can then apply Fekete's lemma to show convergence. However, while we can
combine two $3$AP-free permutations of order $n$ to get a $3$AP-free permutation
of order $2n$ (thus implying the convergence of $\theta(2^n)^{1/2^n}$),
there does not appear to be any simple way to combine $3$AP-free 
permutations of different sizes to obtain a new $3$AP-free permutation.

Using the ideas of this paper, we have the following separation,
which is the largest gap we can prove from the data we have
(namely $\theta(n)$ values up to $n=200$):
We have $L_1\ge(2\theta(128))^{1/128}\approx2.28484$ and
$L_{81}\le(21\theta(162))^{1/162}\approx2.23760$, so
\begin{align*}
2.20
<\liminf_{n\to\infty}\theta(n)^{1/n}
<2.23760
<2.28484
<\limsup_{n\to\infty}\theta(n)^{1/n}
<2.33.
\end{align*} 

\section*{Acknowledgements}
While I wrote this paper myself and checked it carefully,
and while I take full responsibility for its correctness, 
key ideas of the proof were generated with the assistance of GPT-5.2 pro.  
I also thank Way Yan Win for helpful comments.

\hyperref[theorem:1]{Theorem 1} and its proof have been formalized in Lean 4, contingent on the validity
of Sharma's bound (\hyperref[lemma:2]{Lemma 2}) and the exact values of $\theta(64)$ and $\theta(75)$;
for details, as well as code for computing $\theta(n)$, see \url{https://github.com/boonsuan/3ap-free}.


\begin{thebibliography}{10}

\bibitem[EJ73]{EJ73}
Robert C. Entringer and Douglas E. Jackson, Elementary problem E2440, 
\textsl{American Mathematical Monthly} \textbf{80}
(1973), 1058

\bibitem[Sim75]{Sim75}
Gustavus J. Simmons, Solution to E2440, 
\textsl{American Mathematical Monthly} \textbf{82} (1975), 76

\bibitem[DEGS77]{DEGS77}
James A. Davis, Robert C. Entringer, Ronald L. Graham, and Gustavus J. Simmons,
``On permutations containing no long arithmetic progressions,''
\textsl{Acta Arithmetica} \textbf{34} (1977), 81--90

\bibitem[LR04]{LR04}
Bruce M. Landman and Aaron Robertson, \emph{Ramsey Theory on the Integers}, 
American Mathematical Society (2004), 196 (see Research Problem 7.12)

\bibitem[Sha09]{Sha09}
Arun Sharma, ``Enumerating permutations that avoid three term arithmetic progressions,'' 
\textsl{The Electronic Journal of Combinatorics} \textbf{16} (2009), \#R63 (15 pages)

\bibitem[LV11]{LV11}
Timothy D. LeSaulnier and Sujith Vijay, ``On permutations avoiding arithmetic progressions,'' 
\textsl{Discrete Mathematics} \textbf{311} (2011), 205--207

\bibitem[CH17]{CH17}
Bill Correll, Jr.~and Randy W. Ho, ``A Note on $3$-Free Permutations,''
\textsl{Integers} \textbf{17} (2017), \#A55 (9 pages)
\end{thebibliography}
\end{document}